\documentclass[10pt,a4paper]{article}
\usepackage{latexsym,amsfonts,amsmath,amssymb,amsthm,mathrsfs,url}
\usepackage[dvipdfmx]{color,graphicx}
\newtheorem{theorem}{Theorem}
\newtheorem{lemma}{Lemma}
\newtheorem{algorithm}{Algorithm}
\newtheorem{remark}{Remark}

\newtheorem{definition}{Definition}
\newtheorem{corollary}{Corollary}

\newcommand{\rd}{\,\mathrm{d}}

\newcommand{\bsx}{\boldsymbol{x}}
\newcommand{\bsy}{\boldsymbol{y}}
\newcommand{\bsz}{\boldsymbol{z}}

\newcommand{\bst}{\boldsymbol{t}}

\newcommand{\nat}{\mathbb{N}}

\newcommand{\RR}{\mathbb{R}}

\allowdisplaybreaks

\begin{document}

\title{On the separability of multivariate functions}

\author{Takashi Goda\thanks{School of Engineering, The University of Tokyo, 
7-3-1 Hongo, Bunkyo-ku, Tokyo 113-8656, Japan
(\tt{goda@frcer.t.u-tokyo.ac.jp})}}

\date{\today}

\maketitle

\begin{abstract}
Separability of multivariate functions alleviates the difficulty in finding a minimum or maximum value of a function such that an optimal solution can be searched by solving several disjoint problems with lower dimensionalities. In most of practical problems, however, a function to be optimized is black-box and we hardly grasp its separability. In this study, we first describe a general separability condition which a function defined over an arbitrary domain satisfies if and only if the function is separable with respect to given disjoint subsets of variables. By introducing an alternative separability condition, we propose a Monte Carlo-based algorithm to estimate the separability of a function defined over unit cube with respect to given disjoint subsets of variables. Moreover, we extend our algorithm to estimate the number of disjoint subsets and the disjoint subsets such that a function is separable with respect to them. Computational complexity of our extended algorithm is function-dependent and varies from linear to exponential in the dimension.
\end{abstract}

\section{Introduction}
\label{intro}

Whether a given multivariate function is separable or not is one of the important measures of the difficulty in optimization. This can be easily understood through the following argument. Suppose that we want to find a minimum value of a real-valued function $f$ which depends on $s$ variables $\bsx=(x_1,\ldots, x_s)$. In what follows, we put $[1:s]=\{1,\ldots,s\}$, and for a subset $u\subseteq [1:s]$, we write $\bsx_u=(x_j)_{j\in u}$ and $-u:=[1:s]\setminus u$. If $f$ is separable with respect to $\bsx_u$  and its complement $\bsx_{-u}$ with some $\emptyset \ne u\subset [1:s]$, that is, if there exist functions $f_1$ and $f_2$ such that $f(\bsx)=f_1(\bsx_u)+f_2(\bsx_{-u})$ holds for any $\bsx$, we can reduce one high-dimensional optimization problem to two disjoint optimization ones with lower dimensionalities. The values of $\bsx_{-u}$ can be fixed while searching for a minimum value of $f_1$, and vice versa. If $f_1$ and $f_2$ are further separable with respect to subsets $\bsx_{v}$ and $\bsx_{w}$ with some $\emptyset \ne v\subset u$ and $\emptyset \ne w\subset -u$, respectively, for instance, we can reduce to four disjoint optimization problems with even lower dimensionalities. As an extreme case, $f$ might be expressed simply as a sum of $s$ univariate functions, i.e., $f(\bsx)=\sum_{j=1}^{s}f_j(x_j)$. Then, the $s$-dimensional optimization problem can be decomposed into $s$ one-dimensional ones. On the other hand, if $f$ is not separable with respect to any subset of variables, we have to search a whole $s$-dimensional space.

The performance of optimization algorithms, especially of heuristics and meta-heuristics, often depend on separability of the function. For instance, as discussed in \cite{Sal96}, the performance of the genetic algorithm deteriorates if we rotate the coordinate of separable benchmark functions, which makes the functions non-separable. Thus, in order to cover a wide class of functions, we generally compose a set of benchmark functions from many separable and non-separable functions for performance comparison of different optimization algorithms, see for instance \cite{GMLH09,LMH11}. What matters in many practical problems, however, is that a function to be optimized is black-box so that we hardly grasp a priori its separability. If the function is separable, we cannot exploit the advantage of the algorithms which perform better for non-separable functions. Otherwise if the function is non-separable, we should avoid the use of the algorithms which perform well only for separable functions. Therefore, we would claim that the separability of the function to be optimized is one of the central issues in choosing a suitable optimization algorithm.

Motivated by the above concern, we investigate the separability of multivariate functions in this study. Our approach is based on the functional decompositions given in the literature, see for instance \cite{ES81,Hoe48,RASS99,Sob90}. These decompositions were generalized by Kuo et al. \cite{KSWW10}. After introducing the preliminaries on those decompositions in the next section, we first derive a general separability condition which a function defined on an arbitrary domain satisfies if and only if that function is separable with respect to given disjoint subsets of variables in Section~\ref{sec:general}. As special cases, it includes the conditions for a function to be separable with respect to one subset of variables and its complement, or to be separable with respect to all the variables. In order to construct a computable algorithm to estimate the separability, we derive an alternative separability condition in Section~\ref{sec:separable}, which is valid for functions in $L^2([0,1]^s)$. Using this alternative condition, we propose a Monte Carlo-based algorithm for the separability estimation. Moreover, we extend our proposed algorithm to estimate the number of disjoint subsets and the disjoint subsets themselves such that a function is separable with respect to them. We show that computational complexity of our extended algorithm is function-dependent and varies from linear to exponential in the dimension. We conclude this paper with numerical experiments in Section~\ref{sec:numer}.
\section{Background and notation}
\label{background}

\subsection{General decomposition formula}
\label{back:1}

Let us consider a decomposition of an $s$-variate function $f \in F$, where $F$ is a linear space of real-valued functions defined on a domain $D\subseteq \RR^s$, into the following form
\[  f = \sum_{u\subseteq [1:s]}f_u .\]
We note that the right-hand side consists of $2^s$ terms with each term $f_u$ depending only on the subset of variables $\bsx_u$. According to \cite[Theorem~2.1]{KSWW10}, $f_u$ can be generally expressed as
\begin{align}\label{eq:decomp1}
  f_u = \left( \prod_{j\in u}(I-P_j)\right) P_{-u}(f) ,
\end{align}
where $\{ P_j: j=1,\ldots, s\}$ is a set of commuting projections on $F$ defined on the domain $D$ such that $P_j(f)$ does not depend on $x_j$ and that $P_j(f)=f$ if $f$ does not depend on $x_j$. Further, we define $P_u:= \prod_{j\in u}P_j$ for $\emptyset \neq u\subseteq [1:s]$ and denote by $P_{\emptyset}:=I$ the identity operator. We can rewrite (\ref{eq:decomp1}) into the following recursive relation
\begin{align}\label{eq:decomp2}
  f_u := P_{-u}(f)-\sum_{v\subset u}f_v ,
\end{align}
where, for $u=\emptyset$, we define
\[  f_{\emptyset} := P_{[1:s]}(f) .\]
Note that $f_{\emptyset}$ is a constant since it does not depend on any $x_j$.

We show two important examples of $P_j$. One is called \textit{anchored decomposition}, see for instance \cite{RASS99,Sob03}, which fixes $x_j$ at some $t_j$
\begin{align}\label{eq:proj_anchor}
  P_j(f)(\bsx) = f(x_1,\ldots, x_{j-1},t_j,x_{j+1},\ldots, x_s) .
\end{align}
where the anchor $\bst=(t_1,\ldots, t_s)$ lies in $D$. The other example with $D=[0,1]^s$ is called \textit{analysis of variance (ANOVA) decomposition}, see for instance \cite{ES81,Hoe48,Sob90}, which integrates out $x_j$
\begin{align}\label{eq:proj_anova}
  P_j(f)(\bsx) = \int_{0}^{1}f(x_1,\ldots, x_{j-1},t_j,x_{j+1},\ldots, x_s)\rd t_j .
\end{align}

The latter has often been used in the context of global sensitivity analysis, which measures the relative importance of each subset of variables on the variance of function, see for instance \cite{CMO97,Primer,Sob90,Sob01}. Since we also use this decomposition in this study, the next subsection is devoted to explaining it in more detail.

\subsection{ANOVA decomposition and Sobol' indices}
\label{back:2}
In what follows, for a subset $u\subset [1:s]$, we denote the cardinality of $u$ by $|u|$. For any function $f\in L^2([0,1]^s)$, by using (\ref{eq:decomp2}) and (\ref{eq:proj_anova}), each term $f_u$ can be obtained as
\[  f_u(\bsx_u) = \int_{[0,1]^{s-|u|}}f(\bsx)\rd \bsx_{-u} -\sum_{v\subset u}f_v(\bsx_v) ,\]
where, for $u=\emptyset$, we have
\[  f_{\emptyset} = \int_{[0,1]^s}f(\bsx) \rd\bsx ,\]
which is simply the integral of $f$. This decomposition satisfies the following important properties
\[  \int_{0}^{1}f_u(\bsx_u)\rd x_j=0 ,\]
for any $j\in u$ if $|u|>0$, and
\[ \int_{[0,1]^s}f_u(\bsx_u)f_v(\bsx_v) \rd \bsx=0 ,\]
for $u,v\subseteq [1:s]$ if $u\ne v$. The former can be proved by induction on $|u|$. The latter immediately follows from the former by considering the integration with respect to $x_j$ for any $j\in (u\cup v)\setminus(u\cap v)$. Using this decomposition and its properties, the variance of $f$, denoted by $\sigma^2$, is expressed by
\begin{align*}
\sigma^2 & = \int_{[0,1]^s}f^2(\bsx) \rd\bsx -\left(\int_{[0,1]^s}f(\bsx) \rd\bsx \right)^2 \\
 & = \int_{[0,1]^s}\sum_{u,v\subseteq [1:s]}f_u(\bsx_u)f_v(\bsx_v) \rd\bsx - f^2_{\emptyset} \\
 & = \sum_{u,v\subseteq [1:s]}\int_{[0,1]^s}f_u(\bsx_u)f_v(\bsx_v) \rd\bsx - f^2_{\emptyset} = \sum_{\emptyset \ne u\subseteq [1:s]}\sigma^2_u ,
\end{align*}
where we have defined
\[ \sigma^2_u:=\int_{[0,1]^{|u|}}f^2_u(\bsx_u) \rd\bsx_u .\]
This equality implies that the subset of variables $\bsx_u$ with largest $\sigma^2_u$ affects most the variance of the function. In other words, the function $f$ is more sensitive to the change of values of $\bsx_u$ with larger $\sigma^2_u$. That is why the ANOVA decomposition plays a central role in global sensitivity analysis.

Sobol' indices were first introduced by Sobol' \cite{Sob90} and have recently been generalized by Owen \cite{Owe13} to measure the relative importance of a subset of variables. For $\emptyset \ne u\subseteq [1:s]$, let us define
\[ \underline{\tau}^2_u := \sum_{\emptyset \ne v\subseteq u}\sigma_v^2 ,\]
and
\[ \overline{\tau}^2_u :=  \sum_{v\cap u\ne 0}\sigma_v^2 .\]
Here, $\underline{\tau}^2_u$ is a sum of $\sigma_v^2$ for $v$ contained in $u$, whereas $\overline{\tau}^2_u$ is a sum of $\sigma_v^2$ for $v$ which touches $u$. It is obvious that we have $0\le \underline{\tau}^2_u \le \overline{\tau}^2_u \le \sigma^2$. We often normalize these quantities by $\underline{\tau}^2_u/ \sigma^2$ and $\overline{\tau}^2_u/ \sigma^2$, respectively. From the definition, we have the following identity
\[ \underline{\tau}^2_{-u}+\overline{\tau}^2_u=\sigma^2 .\]

\section{General separability condition}
\label{sec:general}

In what follows, we consider a partition $\{u_1,\ldots, u_m\}$ of the set $[1:s]$, which satisfies the following properties: $u_j\ne \emptyset$ for $j=1,\ldots, m$,
\[  u_i\cap u_j=\emptyset ,\]
if $i\ne j$, and
\[ \bigcup_{j=1}^{m}u_j=[1:s] .\]
The following theorem gives  a general separability condition which is satisfied for any separable function $f\in F$ with respect to given $m$ disjoint subsets of variables $\bsx_{u_1},\ldots, \bsx_{u_m}$ for a partition $\{u_1,\ldots, u_m\}$. 

\begin{theorem}\label{theorem1}
For $m, s\in \nat$ such that $m\le s$, let $\{u_1,\ldots, u_m\}$ be a partition of $[1:s]$. A function $f\in F$ is separable with respect to $\bsx_{u_1},\ldots, \bsx_{u_m}$ if and only if the following equation holds
\begin{align}\label{eq:sep1}
  \left(\prod_{j=1}^{m}\left( I-P_{-u_j}\right)\right)(f) = 0 .
\end{align}
\end{theorem}

In order to prove Theorem \ref{theorem1}, we need the following lemma.
\begin{lemma}\label{lemma1}
For $m, s\in \nat$ such that $m\le s$, let $\{u_1,\ldots, u_m\}$ be a partition of $[1:s]$. Then we have
\[  \prod_{j=1}^{m}\left( I-P_{-u_j}\right) = I+(m-1)P_{[1:s]}-\sum_{j=1}^{m}P_{-u_j}.\]
\end{lemma}

\begin{proof}
We note that $P_{-u_i}\cdot P_{-u_j}=P_{[1:s]}$ for $i\ne j$ since $u_i$ and $u_j$ are disjoint with each other. By using this fact and the following identity 
\[  \prod_{j=1}^{m}(a_j+b_j) = \sum_{v\subseteq [1:m]}\left( \prod_{j\in -v}a_j\right) \left( \prod_{j\in v}b_j\right) ,\]
we have
\begin{align*}
  \prod_{j=1}^{m}(I-P_{-u_j}) & = \sum_{v\subseteq [1:m]}\left( \prod_{j\in -v}I\right) \left( \prod_{j\in v}\left(-P_{-u_j}\right)\right) \\
  & = I-\sum_{j=1}^{m}P_{-u_j}+\sum_{\substack{v\subseteq [1:m]\\ |v|\ge 2}}(-1)^{|v|}\left( \prod_{j\in v}P_{-u_j}\right) \\
  & = I-\sum_{j=1}^{m}P_{-u_j}+\left(\sum_{\substack{v\subseteq [1:m]\\ |v|\ge 2}}(-1)^{|v|}\right) P_{[1:s]} .
\end{align*}
In the last term, we have
\begin{align*}
  \sum_{\substack{v\subseteq [1:m]\\ |v|\ge 2}}(-1)^{|v|} & = \sum_{v\subseteq [1:m]}(-1)^{|v|}-\sum_{\substack{v\subseteq [1:m]\\ |v|\le 1}}(-1)^{|v|} \\
  & = (1-1)^{m}-1+m .
\end{align*}
Thus the result follows.
\end{proof}

Now we are ready to prove Theorem~\ref{theorem1}.
\begin{proof} (Theorem \ref{theorem1})
As shown in the proof of \cite[Theorem~2.1]{KSWW10}, we have $P_{-u}(f)=\sum_{v\subseteq u}f_{v}$ for any $u\subset [1:s]$. Applying this relation and Lemma~\ref{lemma1}, we have for the left-hand side of (\ref{eq:sep1})
\begin{align*}
  \left(\prod_{j=1}^{m}\left( I-P_{-u_j}\right)\right)(f) & = \left(I+(m-1)P_{[1:s]}-\sum_{j=1}^{m}P_{-u_j}\right)(f) \\
  & = f+(m-1)f_{\emptyset}-\sum_{j=1}^{m}\sum_{v_j\subseteq u_j}f_{v_j} \\
  & = f-\left(f_{\emptyset}+\sum_{j=1}^{m}\sum_{\emptyset \ne v_j\subseteq u_j}f_{v_j}\right) .
\end{align*}
Given that this equals 0 for any $\bsx\in D$, we can rewrite (\ref{eq:sep1}) into
\[  f = f_{\emptyset}+\sum_{j=1}^{m}\sum_{\emptyset \ne v_j\subseteq u_j}f_{v_j} .\]
Since $f_{\emptyset}$ is a constant and $u_1,\ldots, u_m$ are disjoint with each other, this equation implies that $f$ is separable with respect to  $\bsx_{u_1},\ldots, \bsx_{u_m}$. The proof of the reverse direction is trivial. Hence the result follows.
\end{proof}

Our general separability condition (\ref{eq:sep1}) consists only of function $f$ and projections $(P_u)_{u\subseteq [1:s]}$ and does not include any representation in terms of $(f_u)_{u\subseteq [1:s]}$. We would emphasize here that the condition (\ref{eq:sep1}) is not the same as $(I-P_{-u_j})(f)=0$ for at least one of $j$ with $1\le j\le m$, which only gives
\[  f = \sum_{v_j\subseteq u_j}f_{v_j} .\]
Thus, $(I-P_{-u_j})(f)=0$ for some $j$ is just a sufficient condition for $f$ to be separable with respect to $\bsx_{u_1},\ldots, \bsx_{u_m}$. In the following, we describe the separability conditions for two special cases, both of which might be important in practice.

\begin{corollary}\label{cor1}
A function $f\in F$ is separable with respect to $\bsx_u$ and $\bsx_{-u}$ if and only if the following equation holds
\[  \left( I+P_{[1:s]}-P_{u}-P_{-u}\right) (f) = 0.\]
\end{corollary}
It immediately follows by inserting $m=2$, $u_1=u$ and $u_2=-u$ into (\ref{eq:sep1}) and by applying Lemma \ref{lemma1}.

\begin{corollary}\label{cor2}
A function $f\in F$ is separable with respect to all the variables if and only if the following equation holds
\[  \left( I+(s-1)P_{[1:s]}-\sum_{j=1}^{s}P_{-\{j\}}\right) (f) = 0.\]
\end{corollary}
It also immediately follows by inserting $m=s$ and $u_j=\{j\}$ for $j=1,\ldots, s$ into (\ref{eq:sep1}) and by applying Lemma \ref{lemma1}.

\section{Separability estimation of multivariate functions}
\label{sec:separable}

In the previous section, we have shown a general separability condition, which is necessary and sufficient for $f$ to be separable with respect to given disjoint subsets of variables. It is quite difficult, however, to confirm whether a given black-box function $f$ satisfies this condition or not. Hence, in this section, we propose a computational algorithm based on Monte Carlo methods to estimate the separability of $f$. The key ingredient lies in the use of ANOVA decomposition and Sobol' indices. 

For this purpose we need to restrict ourselves to $f\in L^2([0,1]^s)$. In many practical problems, $D\subseteq \RR^s$ can be replaced by $[0,1]^s$ using suitable transformation of variables. For instance, if $f$ is defined on $\RR^s$ and square-integrable with respect to a density function $\rho$ with independent marginal densities $\rho_1,\ldots,\rho_s$, the function
\[ g(x_1,\ldots,x_s) = f (\Phi^{-1}_1(x_1),\ldots, \Phi^{-1}_s(x_s)) \]
is in $L^2((0,1)^s)$, where $\Phi_i$ denotes the cumulative density function of $\rho_i$
\[ \Phi_i(x) = \int_{-\infty}^{x}\rho_i(t)\rd t. \]
and $\Phi_i^{-1}$ denotes its inverse.

The following theorem shows an alternative separability condition for $f\in L^2([0,1]^s)$, which will be used later in proposing a computational algorithm to estimate the separability of $f$.
\begin{theorem}\label{theorem2}
For $m, s\in \nat$ such that $m\le s$, let $\{u_1,\ldots, u_m\}$ be a partition of $[1:s]$. A function $f\in L^2([0,1]^s)$ is separable with respect to $\bsx_{u_1},\ldots, \bsx_{u_m}$ if and only if the following equation holds
\begin{align}\label{eq:seq_anova1}
\sum_{j=1}^{m}\underline{\tau}^2_{u_j}=\sigma^2 .
\end{align}
\end{theorem}

\begin{proof}
From the definition of $\underline{\tau}^2_{u}$, it is possible to rewrite (\ref{eq:seq_anova1}) into
\[ \sum_{j=1}^{m}\sum_{\emptyset \ne v_j\subseteq u_j}\sigma_{v_j}^2=\sum_{\emptyset \ne v\subseteq [1:s]}\sigma_v^2. \]
This equation implies that for any subset $v\subset [1:s]$ which is not a subset of $u_j$ for all $j=1,\ldots, m$, we have $\sigma^2_v=0$ and thus $f_v=0$. Therefore, $f$ can be expressed as
\[  f = f_{\emptyset}+\sum_{j=1}^{m}\sum_{\emptyset \ne v_j\subseteq u_j}f_{v_j} .\]
The proof of the reverse direction is trivial. Hence the result follows.
\end{proof}

Now we introduce the following notation.
\begin{definition}\label{def1}
For $m, s\in \nat$ such that $m\le s$, let $\{u_1,\ldots, u_m\}$ be a partition of $[1:s]$. We define a \emph{separability index} of $f$ with respect to $u_1,\ldots, u_m$ by
\[ \gamma_{u_1,\ldots, u_m}^2=\sigma^2-\sum_{j=1}^{m}\underline{\tau}^2_{u_j} .\]
\end{definition}

It is trivial from the definition that $\gamma_{u_1,\ldots, u_m}^2$ range from 0 to $\sigma^2$. Further, we emphasize that the condition $\gamma_{u_1,\ldots, u_m}^2=0$ is substituted for the condition $\sum_{j=1}^{m}\underline{\tau}^2_{u_j}=\sigma^2$ given in Theorem~\ref{theorem2}. Our goal is to construct an algorithm which estimates $\gamma_{u_1,\ldots, u_m}^2$ of a black-box function $f$ computationally. In order to obtain a computable form for estimation of $\gamma_{u_1,\ldots, u_m}^2$, we use the integral form of $\underline{\tau}^2_u$, see for example \cite{Owe13,Sal02}
\[ \underline{\tau}^2_u=\int_{[0,1]^{2s}}f(\bsx)\left(f(\bsx_u,\bsz_{-u})-f(\bsz)\right)\rd\bsx \rd\bsz ,\]
and that of $\sigma^2$
\[ \sigma^2=\int_{[0,1]^{2s}}f(\bsx)\left(f(\bsx)-f(\bsz)\right)\rd\bsx \rd\bsz ,\]
where the $s$-dimensional vector $(\bsx_u,\bsz_{-u})$ denotes $\bsy=(y_1,\ldots,y_s)$ in which $y_j=x_j$ for $j\in u$ and $y_j=z_j$ for $j\in -u$.
Then, we have
\begin{align*}
\gamma_{u_1,\ldots, u_m}^2 & = \int_{[0,1]^{2s}}f(\bsx)\left(f(\bsx)-f(\bsz)-\sum_{j=1}^{m}\left(f(\bsx_{u_j},\bsz_{-u_j})-f(\bsz) \right)\right)\rd\bsx \rd\bsz \nonumber \\
& = \int_{[0,1]^{2s}}f(\bsx)\left(f(\bsx)+(m-1)f(\bsz)-\sum_{j=1}^{m}f(\bsx_{u_j},\bsz_{-u_j})\right)\rd\bsx \rd\bsz .
\end{align*}

\begin{remark}\label{rem1}
Let $\bsx,\bsz\in [0,1]^s$. Applying the anchored decomposition (\ref{eq:proj_anchor}) with the anchor $\bsz$ to Theorem~\ref{theorem1} and Lemma~\ref{lemma1}, we have
\[ f(\bsx)+(m-1)f(\bsz)-\sum_{j=1}^{m}f(\bsx_{u_j},\bsz_{-u_j}) = 0, \]
if and only if $f$ is is separable with respect to $\bsx_{u_1},\ldots, \bsx_{u_m}$. Hence the integrand of the above expression for $\gamma_{u_1,\ldots, u_m}^2$ is always 0.
\end{remark}
Since the integral can be approximated by using Monte Carlo methods that take the average of function evaluations at random points with equal weights, we propose the following algorithm to estimate $\gamma_{u_1,\ldots, u_m}^2$.

\begin{algorithm}\label{algorithm1}(Estimation of $\gamma_{u_1,\ldots, u_m}^2$)
For $m, s\in \nat$ such that $m\le s$, let $\{u_1,\ldots, u_m\}$ be a partition of $[1:s]$. For $n\in \nat$, we proceed as follows.
\begin{enumerate}
\item Generate $\bsx^{(i)},\bsz^{(i)}\in [0,1]^s$ for $0\le i<n$ randomly and independently.
\item Compute the approximation of $\gamma_{u_1,\ldots, u_m}^2$
\begin{align}\label{approx}
\hat{\gamma}_{u_1,\ldots, u_m}^2=\frac{1}{n}\sum_{i=0}^{n-1}g_{u_1,\ldots, u_m}(\bsx^{(i)},\bsz^{(i)}),
\end{align}
where 
\[ g_{u_1,\ldots, u_m}(\bsx,\bsz):=f(\bsx)\left( f(\bsx)+(m-1)f(\bsz)-\sum_{j=1}^{m}f(\bsx_{u_j},\bsz_{-u_j}) \right) \]
\end{enumerate}
\end{algorithm}

It is obvious that the computational complexity of our algorithm is linear in $m$ and $n$. Furthermore, since the expression in the parenthesis of (\ref{approx}) is zero for any $\bsx^{(i)},\bsz^{(i)}\in [0,1]^s$ as pointed out in Remark~\ref{rem1}, our algorithm ideally yields $\hat{\gamma}_{u_1,\ldots, u_m}^2=0$ exactly when $f$ is separable with respect to $\bsx_{u_1},\ldots, \bsx_{u_m}$. This is not always the case in practice, however, since numerical computation is subject to round-off or truncation error. 

To address this issue one can consider the following statistical hypothesis testing. Note that $\hat{\gamma}_{u_1,\ldots, u_m}^2$ is an unbiased Monte Carlo estimator of $\gamma_{u_1,\ldots, u_m}^2$. Hence, if the variance of the bi-variate function $g_{u_1,\ldots, u_m}$, denoted by $\tilde{\sigma}_{u_1,\ldots, u_m}^2$, is finite, then the central limit theorem holds, i.e., $\sqrt{n}\left(\hat{\gamma}_{u_1,\ldots, u_m}^2-\gamma_{u_1,\ldots, u_m}^2\right)$ converges in distribution to the normal distribution $N(0,\tilde{\sigma}_{u_1,\ldots, u_m}^2)$ as $n\to \infty$. Here the finiteness of $\tilde{\sigma}_{u_1,\ldots, u_m}^2$ is ensured for any $f\in L^4([0,1]^s)$. 

Now let $H_0\colon \gamma_{u_1,\ldots, u_m}^2=0$ be the null hypothesis and $H_1\colon \gamma_{u_1,\ldots, u_m}^2>0$ be the alternative hypothesis. Under the null hypothesis, $\sqrt{n}\hat{\gamma}_{u_1,\ldots, u_m}^2$ is asymptotically normally distributed and its asymptotic variance can be estimated by the sample variance 
\[ \tilde{s}_{u_1,\ldots, u_m}^2 = \frac{1}{n-1}\sum_{i=0}^{n-1}\left( g(\bsx^{(i)},\bsz^{(i)})-\hat{\gamma}_{u_1,\ldots, u_m}^2\right)^2 \]
with $\bsx^{(i)},\bsz^{(i)}\in [0,1]^s$ for $0\le i<n$ generated in the first step of Algorithm~\ref{algorithm1}.
Thus it is possible to construct a test statistic
\[ T_{u_1,\ldots, u_m} =\sqrt{n}\frac{\hat{\gamma}_{u_1,\ldots, u_m}^2}{\tilde{s}_{u_1,\ldots, u_m}}.\]
We reject the null hypothesis with a significance level $\alpha$ if $T_{u_1,\ldots, u_m}>Q_{1-\alpha}$, where $Q_{1-\alpha}$ denotes the $(1-\alpha)$-quantile of the standard normal distribution.

Here we note that if $f$ is separable, the variance $\tilde{\sigma}_{u_1,\ldots, u_m}^2$ is 0. This may lead to numerical instability in computing $T$ since it involves computing the ratio of two extremely small numbers $\hat{\gamma}_{u_1,\ldots, u_m}^2$ and $\tilde{s}_{u_1,\ldots, u_m}$. In practice, $\tilde{s}_{u_1,\ldots, u_m}$ in the definition of $T_{u_1,\ldots, u_m}$ can be replaced by $\max(\tilde{s}_{u_1,\ldots, u_m},\varepsilon)$ with a user-specified small parameter $\varepsilon>0$.

So far we have discussed how to estimate the separability of functions with respect to given $\bsx_{u_1},\ldots, \bsx_{u_m}$. In order to search for a partition $\{u_1,\ldots, u_m\}$ itself such that $\gamma_{u_1,\ldots, u_m}^2$ is zero, we need to try so many possible candidates of $\{u_1,\ldots, u_m\}$ for $m=2,\ldots, s$. For making a systematic search for such $m$ and $u_1,\ldots, u_m$, we use the following lemma.

\begin{lemma}\label{lemma2}
That $f$ is separable with respect to $\bsx_{u_1},\ldots, \bsx_{u_m}$ is equivalent to that $f$ is separable with respect to $\bsx_{u_j}$ and $\bsx_{-u_j}$ for all $j=1,\ldots, m$.
\end{lemma}

Since this lemma is trivial, we omit the proof. This lemma implies that it is sufficient to search $u$ one-by-one whose value of $\gamma_{u,-u}^2$ is zero without $u$ touching the already found ones. Moreover, due to symmetry of $u$ and $-u$, the overall search space of $u$ can be reduced to $\emptyset \ne u\subseteq [1:s-1]$ and we can simply write $\gamma_u^2:=\gamma_{u,-u}^2$. Based on these observations, we proceed the search in the following order 
\begin{align*}
u & = \{1\}, \\
u & = \{2\}, \{1,2\}, \\
u & = \{3\}, \{1,3\}, \{2,3\}, \{1,2,3\},\\
  & \vdots \\
u & = \{s-1\}, \{1,s-1\}, \ldots, \{1,\ldots, s-1\} .
\end{align*}
If $\gamma_{u}^2$ turns out to be zero (or, if the null hypothesis $H_0\colon \gamma_{u_1,\ldots, u_m}^2=0$ is not rejected) during this process, we can omit from the remaining candidates every subset that touches at least one component of $u$.

For example, if $s=5$ and $f$ is separable with respect to $x_{1},\bsx_{\{2,4\}},\bsx_{\{3,5\}}$, we proceed the search as follows.
\begin{align*}
u & = \{1\}^*, \\
u & = \{2\}, \\
u & = \{3\},\{2,3\} \\
u & = \{4\},\{2,4\}^* ,
\end{align*}
where $*$ means that the corresponding subset of variables is found to be separable. Consequently, we obtain $u_1=\{1\},u_2=\{2,4\}$.  From Lemma \ref{lemma2}, we have $m=3$ and $u_3=\{3,5\}$.

Hence, our extended algorithm to estimate the number of disjoint subsets $m$ and the disjoint subsets themselves $u_1,\ldots, u_m$ is given as follows.

\begin{algorithm}\label{algorithm2}(Estimation of $m$ and $u_1,\ldots, u_m$)
For $s,n\in \nat$, we proceed as follows.
\begin{enumerate}
\item Generate $\bsx^{(i)},\bsz^{(i)}\in [0,1]^s$ for $0\le i<n$ randomly and independently, and set $r=m=1$.
\item For each subset $v$ such that $v\subseteq [1:r-1]\setminus \bigcup_{j=1}^{m-1}u_j$, compute $\hat{\gamma}_{v\cup\{r\}}^2$ according to (\ref{approx}). If one finds $v$ such that $\hat{\gamma}_{v\cup\{r\}}^2=0$, set $u_m=v\cup\{r\}$ and $m=m+1$.
\item Set $r=r+1$. If $r<s$, go to step 2.
\end{enumerate}
\end{algorithm}

The second step of Algorithm~\ref{algorithm2} can be replaced as follows:
\begin{enumerate}
\renewcommand{\labelenumi}{\arabic{enumi}'}
\setcounter{enumi}{1}
\item Set a significance level $\alpha$. For each subset $v$ such that $v\subseteq [1:r-1]\setminus \bigcup_{j=1}^{m-1}u_j$, compute the test statistic $T_{u_1,\ldots, u_m}$. If one finds $v$ such that $T_{u_1,\ldots, u_m}\leq Q_{1-\alpha}$, set $u_m=v\cup\{r\}$ and $m=m+1$.
\end{enumerate}

The computational complexity of our extended algorithm is function-dependent as follows. When $f$ is separable with respect to all the variables, our algorithm searches only $u=\{1\},\ldots,\{s\}$ in this order. Hence the computational complexity is minimized and becomes linear in $s$ and $n$. When $f$ is not separable with respect to any subset of variables, on the other hand, our algorithm searches for all the candidates $\emptyset \ne u\subset [1:s-1]$ so that the computational complexity is maximized. Since the cardinality of $u$ such that $\emptyset \ne u\subset [1:s-1]$ is $2^{s-1}-2$, the computational complexity becomes exponential in $s$.

From this point, Algorithm~\ref{algorithm2} should work for small $s$ but becomes infeasible as $s$ increases. How to overcome this drawback is open for further research. At this moment, for large $s$, Algorithm \ref{algorithm1} with $m=s$ and $u_j=\{j\}$ for $j=1,\ldots,s$ will be of use as an initial screening to estimate the separability with respect to all the variables at one time, which can be done with the computational complexity linear in $s$.

\section{Numerical experiments}\label{sec:numer}
Finally we conduct simple numerical experiments to illustrate how our algorithms work. Here we focus on applying Algorithm~\ref{algorithm1} and the statistical hypothesis testing for estimating the separability of functions with respect to all the variables. Let us consider the following two test functions:
\[ f_{1}(x_1,\ldots,x_s)  = \sum_{j=1}^{s}\left( x_j^2-10 \cos(2\pi x_j)+10 \right) \]
defined over $[-5.12,5.12]^s$, known as Rastrigin function, and 
\[ f_{2}(x_1,\ldots,x_s) = \sum_{j=1}^{s-1}\left( 100(x_j^2-x_{j+1})^2+(x_j-1)^2\right) \]
defined over $[-2,2]^s$, known as Rosenbrock function. A simple linear transformation enables us to transform these to functions on $[0,1]^s$ and we shall use such transformations without further notice. The minimum values of $f_1$ and $f_2$ are both 0, which are attained at $(x_1,\ldots,x_s)=(0,\ldots,0)$ and $(x_1,\ldots,x_s)=(1,\ldots,1)$, respectively. It is obvious that $f_1$ is separable, whereas $f_2$ is not.

Table~\ref{tb:result1} shows the results for $s=2$ with sample sizes $n=10^3,10^4,10^5,10^6$. Note that, for the hypothesis testing, we use a modified test statistic
\[ T_{u_1,\ldots,u_m}'=\sqrt{n}\frac{\hat{\gamma}_{u_1,\ldots, u_m}^2}{\max(\tilde{s}_{u_1,\ldots,u_m},\varepsilon)} \]
with $\varepsilon=10^{-12}$. As expected, the values of $\hat{\gamma}_{1,2}^2$ for $f_1$ are extremely small, while not for $f_2$. It is instructive to see that $\hat{\gamma}_{1,2}^2$ is not exactly 0 even for such a simple separable function $f_1$. Thus the statistical testing is helpful in this regard. The values of $T_{1,2}'$ for $f_1$ are small enough that the null hypothesis $H_0\colon \gamma_{1,2}^2=0$ is not rejected, for instance, with a significance level $\alpha=0.05$. On the other hand, the value of $T_{1,2}'$ for $f_2$ becomes larger as $n$ increases. In fact, it is expected from the definitions that the test statistic $T_{u_1,\ldots,u_m}$ or $T_{u_1,\ldots,u_m}'$ for non-separable functions should diverge with order $\sqrt{n}$, which is supported by this numerical result. Again with a significance level $\alpha=0.05$, the null hypothesis is rejected for all $n$. As shown in Table~\ref{tb:result2}, a similar result is obtained even for the high-dimensional case where $s=50$.

\begin{table}[!ht]
\begin{center}
\begin{tabular}{c|r|r|r}
& $n$ & $f_1$ & $f_2$ \\ \hline
                                & $10^3$ & $1.90\times 10^{-15}$ & $8.48\times 10^4$ \\
$\hat{\gamma}_{1,2}^2$ & $10^4$ & $9.36\times 10^{-16}$ & $7.18\times 10^4$ \\
                                & $10^5$ & $6.66\times 10^{-17}$ & $7.32\times 10^4$ \\ 
                                & $10^6$ & $6.25\times 10^{-17}$ & $7.54\times 10^4$ \\ \hline
       & $10^3$ & $0.0602$ & $4.33$ \\
$T_{1,2}'$ & $10^4$ & $0.0936$ & $11.95$ \\
       & $10^5$ & $0.0211$ & $39.78$ \\
       & $10^6$ & $0.0626$ & $127.60$ \\ \hline
\end{tabular}
\caption{Separability estimation for $f_1$ and $f_2$ with $s=2$ \label{tb:result1}}
\end{center}
\end{table}

\begin{table}[!ht]
\begin{center}
\begin{tabular}{c|r|r|r}
& $n$ & $f_1$ & $f_2$ \\ \hline
                                & $10^3$ & $9.84\times 10^{-11}$ & $6.36\times 10^6$ \\
$\hat{\gamma}_{1,\ldots,50}^2$ & $10^4$ & $2.50\times 10^{-12}$ & $3.28\times 10^6$ \\
                                & $10^5$ & $3.76\times 10^{-12}$ & $3.57\times 10^6$ \\ 
                                & $10^6$ & $6.03\times 10^{-12}$ & $3.68\times 10^6$ \\ \hline
       & $10^3$ & $0.336$ & $2.28$ \\
$T_{1,\ldots,50}'$ & $10^4$ & $0.270$ & $3.68$ \\
       & $10^5$ & $0.127$ & $12.56$ \\
       & $10^6$ & $0.645$ & $40.94$ \\ \hline
\end{tabular}
\caption{Separability estimation for $f_1$ and $f_2$ with $s=50$ \label{tb:result2}}
\end{center}
\end{table}

\section*{Acknowledgments}
This work was supported by JSPS Grant-in-Aid No.~24-4020 and No.~15K20964.


\begin{thebibliography}{14}
\bibitem{CMO97} R.~E. Caflisch, W. Morokoff, A.~B. Owen, Valuation of mortgage backed securities using Brownian bridges to reduce effective dimension, J. Comput. Finance 1 (1997) 27--46.
\bibitem{ES81} B. Efron, C. Stein, The jackknife estimate of variance, Ann. Statist. 9 (1981) 586--596.
\bibitem{GMLH09} S. Garc{\'i}a, D. Molina, M. Lozano, F. Herrera, A study on the use of non-parametric tests for analyzing the evolutionary algorithms' behaviour: a case study on the CEC' 2005 Special Session on Real Parameter Optimization, J. Heuristics. 15 (2009) 617--644.
\bibitem{Hoe48} W. Hoeffding, A class of statistics with asymptotically normal distribution, Ann. Math. Statistics 19 (1948) 293--325.
\bibitem{KSWW10} F.~Y. Kuo, I.~H. Sloan, G.~W. Wasilkowski, H. Wo{\'z}niakowski, On decompositions of multivariate functions, Math. Comp. 79 (2010) 953--966.
\bibitem{LMH11} M. Lozano, D. Molina, F. Herrera, Editorial scalability of evolutionary algorithms and other metaheuristics for large-scale continuous optimization problems, Soft Comput. 15 (2011) 2085--2087.
\bibitem{Owe13} A.~B. Owen, Variance components and generalized Sobol' indices, SIAM/ASA J. Uncertainty Quantification 1 (2013) 19--41.
\bibitem{RASS99} H. Rabitz, O.~F. Alis, J. Shorter, K. Shim, Efficient input-output model representation, Comput. Phys. Commun. 117 (1999) 11--20.
\bibitem{Sal96} R. Salomon, Re-evaluating genetic algorithm performance under coordinate rotation of benchmark functions. A survey of some theoretical and practical aspects of genetic algorithms, BioSystems 39 (1996) 263--278.
\bibitem{Sal02} A. Saltelli, Making best use of model evaluations to compute sensitivity indices, Comput. Phys. Commun. 145 (2002) 280--297.
\bibitem{Primer} A. Saltelli, M. Ratto, T. Andres, F. Campolongo, J. Cariboni, D. Gatelli, M. Saisana, S. Tarantola, Global Sensitivity Analysis. The Primer, John Wiley and Sons, New York, 2008.
\bibitem{Sob90} I.~M. Sobol', Sensitivity estimates for nonlinear mathematical models, Math. Model. Comput. Exp. 1 (1993) 407--414.
\bibitem{Sob01} I.~M. Sobol', Global sensitivity indices for nonlinear mathematical models and their Monte Carlo estimates, Math. Comput. Simul. 55 (2001) 271--280.
\bibitem{Sob03} I.~M. Sobol', Theorems and examples on high dimensional model representation, Reliab. Eng. Syst. Saf. 79 (2003) 187--193.
\end{thebibliography}
\end{document}